\documentclass[12pt]{amsart}       
\usepackage{txfonts}
\usepackage{extarrows}
\usepackage{amssymb}
\usepackage{eucal}
\usepackage{bbm}
\usepackage{graphicx}
\usepackage{amsmath}
\usepackage{amscd}
\usepackage[all]{xy}           
\usepackage{amsfonts,latexsym}
\usepackage{xspace}
\usepackage{epsfig}
\usepackage{float}
\usepackage{color}
\usepackage{shuffle}
\usepackage{fancybox}
\usepackage{colordvi}
\usepackage{multicol}
\usepackage{colordvi}
\usepackage{mathrsfs}
\usepackage{ifpdf}
\ifpdf
  \usepackage[colorlinks,final,backref=page,hyperindex]{hyperref}
\else
  \usepackage[colorlinks,final,backref=page,hyperindex,hypertex]{hyperref}
\fi
\usepackage[active]{srcltx} 
\usepackage{tikz}
\usepackage{graphicx}

\newtheorem{theorem}{Theorem}[section]
\newtheorem{proposition}[theorem]{Proposition}
\newtheorem{lemma}[theorem]{Lemma}

\newtheorem{prop-def}{Proposition-Definition}[section]
\newtheorem{coro-def}{Corollary-Definition}[section]

\theoremstyle{definition}
\newtheorem{definition}[theorem]{Definition}
\newtheorem{remark}[theorem]{Remark}

\newcommand{\nc}{\newcommand}
\nc{\tred}[1]{\textcolor{red}{#1}}
\nc{\tblue}[1]{\textcolor{blue}{#1}}
\nc{\tgreen}[1]{\textcolor{green}{#1}}
\nc{\tpurple}[1]{\textcolor{purple}{#1}}
\nc{\btred}[1]{\textcolor{red}{\bf #1}}
\nc{\btblue}[1]{\textcolor{blue}{\bf #1}}
\nc{\btgreen}[1]{\textcolor{green}{\bf #1}}
\nc{\btpurple}[1]{\textcolor{purple}{\bf #1}}
\nc{\NN}{{\mathbb N}}
\nc{\ncsha}{{\mbox{\cyr X}^{\mathrm NC}}} \nc{\ncshao}{{\mbox{\cyr
X}^{\mathrm NC}_0}}

\newcommand{\delete}[1]{}

\nc{\mlabel}[1]{\label{#1}}
\nc{\mcite}[1]{\cite{#1}}
\nc{\mref}[1]{\ref{#1}}
\nc{\meqref}[1]{\eqref{#1}}
\nc{\mbibitem}[1]{\bibitem{#1}}

\delete{
\nc{\mlabel}[1]{\label{#1}{\hfill \hspace{1cm}{\bf{{\ }\hfill(#1)}}}}
\nc{\mcite}[1]{\cite{#1}{{\bf{{\ }(#1)}}}}
\nc{\mref}[1]{\ref{#1}{{\bf{{\ }(#1)}}}}
\nc{\meqref}[1]{\eqref{#1}{{\bf{{\ }(#1)}}}}
\nc{\mbibitem}[1]{\bibitem[\bf #1]{#1}}
}
\nc{\sha}{{\mbox{\cyr X}}}  
\newfont{\scyr}{wncyr10 scaled 550}
\nc{\ssha}{\mbox{\bf \scyr X}}
\nc{\shap}{{\mbox{\cyrs X}}} 
\nc{\shpr}{\diamond}    
\nc{\shp}{\ast} \nc{\shplus}{\shpr^+}
\nc{\shprc}{\shpr_c}    
\nc{\dep}{\mrm{dep}} \nc{\lc}{\lfloor} \nc{\rc}{\rfloor}
\nc{\db}{\leq_{\rm db}} \nc{\bfk}{{\bf k}}

\nc{\cala}{{\mathcal A}} \nc{\calb}{{\mathcal B}}
\nc{\calc}{{\mathcal C}}
\nc{\cald}{{\mathcal D}} \nc{\cale}{{\mathcal E}}
\nc{\calf}{{\mathcal F}} \nc{\calg}{{\mathcal G}}
\nc{\calh}{{\mathcal H}} \nc{\cali}{{\mathcal I}}
\nc{\call}{{\mathcal L}} \nc{\calm}{{\mathcal M}}
\nc{\caln}{{\mathcal N}} \nc{\calo}{{\mathcal O}}
\nc{\calp}{{\mathcal P}} \nc{\calr}{{\mathcal R}}
\nc{\cals}{{\mathcal S}} \nc{\calt}{{\mathcal T}}
\nc{\calu}{{\mathcal U}} \nc{\calw}{{\mathcal W}} \nc{\calk}{{\mathcal K}}
\nc{\calx}{{\mathcal X}} \nc{\CA}{\mathcal{A}}

\nc{\fraka}{{\mathfrak a}} \nc{\frakA}{{\mathfrak A}}
\nc{\frakb}{{\mathfrak b}} \nc{\frakB}{{\mathfrak B}}
\nc{\frakc}{{\mathfrak c}}
\nc{\frakD}{{\mathfrak D}} \nc{\frakF}{\mathfrak{F}}
\nc{\frakf}{{\mathfrak f}} \nc{\frakg}{{\mathfrak g}}
\nc{\frakH}{{\mathfrak H}} \nc{\frakL}{{\mathfrak L}}
\nc{\frakM}{{\mathfrak M}} \nc{\bfrakM}{\overline{\frakM}}
\nc{\frakm}{{\mathfrak m}} \nc{\frakP}{{\mathfrak P}}
\nc{\frakN}{{\mathfrak N}} \nc{\frakp}{{\mathfrak p}}
\nc{\frakS}{{\mathfrak S}} \nc{\frakT}{\mathfrak{T}}
\nc{\frakX}{{\mathfrak X}}

\font\cyr=wncyr10 \font\cyrs=wncyr7
\nc{\li}[1]{\textcolor{blue}{Nan:#1}}
\nc{\lir}[1]{\textcolor{red}{Li:#1}}
\nc{\yi}[1]{\textcolor{blue}{Yi: #1}}
\nc{\xing}[1]{\textcolor{purple}{Xing:#1}}
\nc{\revise}[1]{\textcolor{red}{#1}}
\nc{\nan}[1]{\textcolor{blue}{Nan:#1}}

\numberwithin{equation}{section}
\nc{\RR}{\mathbb{R}}
\nc{\X}{{\bf X}}
\nc{\E}{{\bf E}}
\nc{\x}{\mathbb{X}}
\nc{\C}{\mathcal{C}^{\alpha}}
\nc{\D}{\mathcal{D}^{\alpha}}
\nc{\CC}{\mathcal{C}_{\X}^{\alpha}}
\nc{\f}{\varphi}
\nc{\al}{\alpha}
\nc{\lbar}{\overline}
\nc{\HA}{\mathbb{S}}
\nc{\ha}{\mathcal{S}}

\nc{\V}{V} \nc{\pro}{\otimes}
\nc{\tng}{T^{\le N}(V)^{g}} \nc{\tn}{T^{\le N}(V)}
\nc{\ttg}{T^{\le 3}(V)^{g}}
\nc{\ZZ}{\mathbb{Z}} \nc{\etree}{1}
\nc{\xx}{\mathcal{X}}
\nc{\RP}{{\mathcal{D}}^{\alpha}([0, T]^2, V)}
\nc{\Y}{{\bf Y}}\nc{\id}{\text{id}} \nc{\Id}{\text{Id}}\nc{\Z}{{\bf Z}}
\nc{\sym}{\operatorname{Sym}} \nc{\tri}{\operatorname{Sym}}

\begin{document}

\title[Jump It\^o-type formula with arbitrary regularity]{Jump It\^o-type formula with arbitrary regularity}
%
%
\author{Nannan Li}
\address{School of Mathematics and Statistics, Lanzhou University
Lanzhou, 730000, China
}
\email{linn2024@lzu.edu.cn}

\author{Xing Gao$^{*}$}\thanks{*Corresponding author}
\address{School of Mathematics and Statistics, Lanzhou University
Lanzhou, 730000, China; Gansu Provincial Research Center for Basic Disciplines of Mathematics
and Statistics, Lanzhou, 730070, China
}
\email{gaoxing@lzu.edu.cn}
\begin{abstract}
We establish an It\^o-type formula for finite $p$-variation paths with jumps for arbitrary $p\geq 1$. The formula is stated in a fully pathwise form and separates the reduced rough integral from explicit left- and right-jump correction terms. 
In the c\`adl\`ag case, only the left-jump correction remains, while in the continuous case, both jump correction terms vanish and the formula recovers the corresponding
continuous arbitrary-regularity change-of-variable formula. 
The proof is based on the reduced rough path framework and a refinement Riemann--Stieltjes convergence criterion adapted to discontinuous paths. 
This approach allows us to handle the higher-order Taylor expansions required for large values of $p$ and to control the interaction between rough increments and discrete jumps. 

As applications, we derive It\^o-type formulas for stochastic processes whose sample paths have finite $p$-variation, including pure-jump models and mixed fractional Brownian--jump signals. The latter class includes cases with Hurst parameter $H\leq 1/3$, which fall outside the regime $2\leq p<3$. We also obtain chain-rule identities for nonlinear observables of c\`adl\`ag finite-$p$-variation solutions of random differential equations with jumps, together with a pathwise log-wealth decomposition.
\end{abstract}

\makeatletter
\@namedef{subjclassname@2020}{\textup{2020} Mathematics Subject Classification}
\makeatother
\subjclass[2020]{
60L20, 
60H99, 
60L90. 
}

\keywords{It\^o-type formula, reduced rough path, rough integral, jump path. }

\maketitle

\tableofcontents

\setcounter{section}{0}

\allowdisplaybreaks

\section{Introduction}\label{sec:introduction}

\subsection{Motivation}\label{subsec:intro_motivation}

The It\^o formula is one of the central tools of stochastic analysis. It extends the classical chain rule to irregular stochastic dynamics and provides the basic mechanism for transforming functions of stochastic processes. In its classical form, for a Brownian motion $B$ and a twice continuously differentiable function $f$, the formula contains the additional second-order correction term
\[
df(B_t)=f'(B_t)\,dB_t+\frac12 f''(B_t)\,dt,
\]
which can be expressed in the following integral form:
\begin{equation*}
f(B_T) - f(B_0) = \int_0^T f'(B_t) dB_t + \frac{1}{2} \int_0^T f''(B_t) dt.
\end{equation*}
Here the first integral is understood in the It\^o sense \cite{Ito}. This identity has become a fundamental ingredient in the theory of semimartingales.
From this viewpoint, the classical It\^o formula may be regarded as a stochastic change-of-variable formula. In the present paper, by an It\^o-type formula we mean a pathwise change-of-variable identity, formulated through the reduced rough integral, in which the correction terms are determined by the roughness and the jumps of the driving path.
However, many signals appearing in modern probability, mathematical finance, physics, and stochastic modelling are not adequately described by the classical
semimartingale framework.

Some of them have low H\"older regularity, while others have discontinuities caused by shocks, impulses, or rare events. These features motivate the development of change-of-variable formulas which are pathwise and which do not depend on a probabilistic martingale structure.

Rough path theory, introduced by Lyons \cite{Ly98}, provides a powerful framework for such problems. By enriching an irregular path with suitable higher-order information, rough path theory makes it possible to define integration and differential equations driven by signals far beyond the range of classical calculus. A pathwise It\^o-type formula for rough signals is therefore a natural and important problem.

\subsection{Background and related work}\label{subsec:intro_related_work}

There are two main lines of research related to the present paper. The first line concerns continuous rough paths. Hairer and Kelly studied It\^o-type formulas using branched rough paths \cite{HK15,Ke12}. Friz and Hairer presented a reduced rough path approach in the regime of continuous paths with H\"older regularity greater than $1/3$ \cite{FH20}. Cont and Perkowski developed a pathwise integration theory and change-of-variable formulas for continuous finite $p$-variation paths with arbitrary $p\geq 1$ using reduced rough paths \cite{CP19}. Further developments have been obtained through quasi-geometric rough paths \cite{Be23}, planarly branched rough paths \cite{LG25}, reduced rough paths in lower regularity regimes \cite{LG26}, and recent refinements of branched It\^o formulas \cite{BFT26}.

The second line concerns paths with jumps. Discontinuous paths are essential in applications, since they naturally model sudden shocks and discrete events. In this direction, Friz and Zhang developed a theory of rough differential equations driven by rough paths with jumps and obtained an It\^o-type formula for finite $p$-variation paths with jumps in the regime $2\leq p<3$ \cite{FZ18}. In this range, the second-order Taylor expansion is sufficient.

The arbitrary $p$ case for jump paths is more delicate. When $p$ becomes larger, the path may be rougher, and Taylor expansions up to order $\lfloor p\rfloor$ are required. At the same time, the discontinuities of the path create additional correction terms. Thus, one has to control both the high-order rough increments and their interaction with left and right jumps. This is the gap addressed in the present paper.

\subsection{Aim and main result}\label{subsec:intro_aim}

The purpose of this paper is to prove an It\^o-type formula for finite $p$-variation paths with jumps for arbitrary $p\geq 1$. Our formula is fully pathwise and is formulated in terms of the reduced rough integral. 
It extends the jump It\^o-type formula in~\cite{FZ18} from the range $2\leq p<3$ to all $p\geq 1$, and it extends the continuous pathwise change-of-variable formula in ~\cite{CP19} to discontinuous paths.

Let $n=\lfloor p\rfloor$. The identity proved in Theorem~\ref{thm:main} takes the form
\begin{align*}
F(X_T)-F(X_0)
={}&
\int_0^T DF(X_t)\,dX_t  
+\sum_{0<t\leq T}
\left(
F(X_t)-F(X_{t-})
-\sum_{k=1}^{n}\frac{1}{k!}D^kF(X_{t-})(\Delta_t^-X)^{\otimes k}
\right) \\
&+
\sum_{0\leq t<T}
\left(
F(X_{t+})-F(X_t)
-\sum_{k=1}^{n}\frac{1}{k!}D^kF(X_t)(\Delta_t^+X)^{\otimes k}
\right),
\end{align*}
where $\Delta_t^-X=X_t-X_{t-}$ and $\Delta_t^+X=X_{t+}-X_t$. The first term on the right-hand side of the above formula is the reduced rough integral, defined through a refinement Riemann--Stieltjes limit. The remaining two terms are explicit left- and right-jump corrections.

In the c\`adl\`ag case, the right-jump correction vanishes, so the identity contains only the left-jump correction. In the continuous case, both jump correction terms vanish, and the formula recovers the continuous reduced rough path change-of-variable formula.

\subsection{Main contributions}\label{subsec:intro_contributions}

The main contributions of this paper are the following.  

\begin{enumerate}
\item We establish a pathwise It\^o-type formula for finite $p$-variation paths with jumps for every $p\geq 1$. This extends the jump It\^o-type formula in~\cite{FZ18} from the regime $2\leq p<3$ to all $p \ge 1$ (Theorem~\ref{thm:main}).


\item We construct the corresponding reduced rough integral in the presence of jumps (Proposition~\ref{prop:2}). Since mesh-based convergence is not well adapted to discontinuous paths, the integral is formulated through the refinement Riemann--Stieltjes convergence criterion, following the classical idea of refinement integration \cite{Hi38,Hi63} and the sewing argument for jump paths in \cite{FZ18}.

\item We identify the exact jump correction terms in the arbitrary-regularity setting. These correction terms contain all Taylor contributions up to order $\lfloor p\rfloor$ and therefore describe precisely how high-order rough increments interact with discrete jumps.

 \item We show that the jump It\^o-type formula can be applied directly to stochastic models whose sample paths have finite $p$-variation (Proposition~\ref{prop:as_stochastic_formula}). A key example is the mixed fractional Brownian--jump signal
\[
X_t=X_0+B_t^H+J_t,
\]
for which Proposition~\ref{prop:mixed_driver_application} yields an almost sure It\^o-type formula for every $p>\max\{1/H,q\}$ whenever $J$ has finite $q$-variation. In particular, when $H\leq 1/3$, this application lies beyond the regime $2\leq p<3$ covered by previous jump formulas. We also obtain formulas for pure-jump L\'evy-type processes (Proposition~\ref{prop:levy_driver_application}), chain-rule identities for observables of random differential equations with jumps (Proposition~\ref{prop:random_equation_application}), and pathwise log-wealth decompositions in mathematical finance (Proposition~\ref{prop:finance_application}).
\end{enumerate}

\subsection{Applications and scope}\label{subsec:intro_applications}

The applications are an important motivation for the present work. Since Theorem~\ref{thm:main} is deterministic, it can be applied path by path to stochastic processes whose paths have finite $p$-variation. This gives almost sure It\^o-type formulas without requiring the process to be a semimartingale.

This point is particularly useful for mixed models. For example, a signal consisting of a fractional Brownian component and a jump component can be treated within the same formula. When the Hurst parameter is small, the corresponding variation index may satisfy $p>3$, which lies outside the range covered by the jump formula in \cite{FZ18}. Thus, the arbitrary-regularity feature of Theorem~\ref{thm:main} is essential for such models.

The formula also applies naturally to random differential equations with jumps. Once a solution is known to have finite $p$-variation, observables of the solution satisfy the same pathwise It\^o-type formula. In mathematical finance, applying the theorem to a positive wealth process and the logarithmic function gives a pathwise decomposition of cumulative log-return into a reduced rough integral term and an explicit jump correction.

\subsection{Organization of the paper}\label{subsec:intro_organization}

The rest of the paper is organized as follows. In Section~\ref{sec:2}, we recall the necessary background on reduced rough paths, controlled paths, and refinement Riemann--Stieltjes convergence, and we define the reduced rough integral for paths with jumps (Proposition~\ref{prop:2}). In Section~\ref{sec:3}, we prove the main It\^o-type formula for jump paths of finite $p$-variation with arbitrary $p\ge 1$ (Theorem~\ref{thm:main}). In Section~\ref{sec:4}, we present applications to stochastic processes, representative stochastic drivers, random differential equations, and mathematical finance (Propositions~\ref{prop:as_stochastic_formula},~\ref{prop:levy_driver_application},~\ref{prop:mixed_driver_application},~\ref{prop:random_equation_application} and~\ref{prop:finance_application}).

\smallskip

{\bf Notation.} Let $(V, \|\cdot\|)$, $(W, \|\cdot\|)$, and $(U, \|\cdot\|)$ be Banach spaces. Let $X$ be a function from $\Delta_T:=\{(s,t) \mid 0\leq s\leq t \leq T\}$ into $V$. For any $p\in (0, \infty)$, we define finite $p$-variation,
\[\|X\|_{p\text{-var}; [0, T]}:=\Big(\sup_{\Pi}\sum_{[s, t]\in \Pi}\|X_{s, t}\|^p \Big)^{\frac{1}{p}} <\infty,\]
where the sup is taken over all partitions $\Pi$ of $[0, T]$. In particular, if $X:[0, T]\to V$, we define
$$X_{s, t}:=X_t-X_s.$$

For $A\in V$ and $x\in W$, we will use the notation $A = O(x)$ if there exists a constant $C\in \RR$ such that the bound $\|A\|\le C\|x\|$ holds.

\section{Reduced rough paths and integration with jumps}\label{sec:2}
In this section, we first review reduced rough paths and define reduced controlled rough paths with jumps. We then establish the reduced rough integral with jumps. We also recall an integral convergence criterion adapted to jump paths.

\subsection{Reduced rough path}
We review the setting of the symmetric tensor as described in~\cite{CP19}. Define
$$T_N(V) := \underbrace{V\otimes \cdots \otimes V}_{N}$$
as the space of $N$-tensors on $V$. For any permutation $\sigma$, we define a linear operator $\hat{\sigma}: T_N(V) \to T_N(V)$ by its action on simple tensors as follows.
$$\hat{\sigma}(v_1 \otimes v_2 \otimes \dots \otimes v_N) = v_{\sigma(1)} \otimes v_{\sigma(2)} \otimes \dots \otimes v_{\sigma(N)}.$$
$\text{Sym}_N(V)$ is defined as the subspace of fixed points under the action of the permutation group:
$$\text{Sym}_N(V) = \{ T \in T_N(V) \mid \hat{\sigma}(T) = T, \forall \sigma \in \mathfrak{S}_N \},$$
where $\mathfrak{S}_N$ denotes the group of permutations of $\{1, \dots, N\}$.
We set $\text{Sym}_0(V) := \RR$ and define $\mathbb{S}_N(V)$ as the direct sum of $\text{Sym}_k(V)$ for $k = 0, 1, \dots, N$:
\[\mathbb{S}_N(V) := \bigoplus_{k=0}^N \text{Sym}_k(V).\]

We also introduce jump control functions.
\begin{definition}
A \textit{jump control function} is a map $c: \Delta_T \to \mathbb{R}_+$ such that $c(t,t) = 0$ for all $t \in [0,T]$ and such that $c(s,u) + c(u,t) \le c(s,t)$ for all $0 \le s \le u \le t \le T$.
\mlabel{def:2}
\end{definition}
For paths with jumps, the jump control function is usually discontinuous at the jump points.
We shall use the following elementary fact relating finite $p$-variation to jump control functions.
\begin{lemma}
A function $F : [0,T] \to V$ has finite $p$-variation if and only if there exists a control function $c$ with $\|F(t) - F(s)\|^p \le c(s,t)$, and in that case $\|F\|_{p\textup{-var}; [0, T]} \le c(0,T)^{1/p}$.
\mlabel{lem:1}
\end{lemma}

\begin{proof}
For the sufficiency, suppose that there exists a control function $c(s, t)$ such that $\|F(t) - F(s)\|^p \le c(s, t)$.
Then, for any partition $\Pi = \{0 = t_0 < t_1 < \dots < t_n = T\}$,
\begin{equation*}
\sum_{i=0}^{n-1} \|F(t_{i+1}) - F(t_i)\|^p \le \sum_{i=0}^{n-1} c(t_i, t_{i+1}).
\end{equation*}
By Definition~\ref{def:2}, we have
\begin{equation*}
\sum_{i=0}^{n-1} c(t_i, t_{i+1}) \le c(t_0, t_n) = c(0, T).
\end{equation*}
Since this inequality holds for any partition $\Pi$, according to the definition of $p$-variation, we have
$$\|F\|_{p\text{-var}; [0, T]}^p = \sup_{\Pi} \sum_{[t_i, t_{i+1}]\in \Pi} \|F(t_{i+1}) - F(t_i)\|^p \le c(0, T).$$
Taking the $p$-th root on both sides gives
$$\|F\|_{p\text{-var}; [0, T]} \le c(0, T)^{1/p}.$$
Therefore, $F$ has finite $p$-variation.\\

For the necessity, assume that $F$ has finite $p$-variation. Let
$$c(s, t) := \|F\|_{p\text{-var}; [s, t]}^p.$$
We now verify that $c(s, t)$ satisfies the definition of a jump control function. Obviously,
$$c(s, s) = \|F\|_{p\text{-var}; [s, s]}^p = 0.$$

Note that for any $s \le u \le t$, $\|F\|_{p\text{-var}; [s, u]}^p$ is the upper bound of the partition taken over the interval $[s, u]$, and $\|F\|_{p\text{-var}; [u, t]}^p$ is similarly the upper bound of the partition over the interval $[u, t]$. For any partition $\Pi_1$ of $[s, u]$ and any partition $\Pi_2$ of $[u, t]$, their union $\Pi_1 \cup \Pi_2$ is a partition of $[s, t]$ that contains the point $u$.
Let $\Pi$ be any partition of $[0, T]$, then
$$\sum_{[t_i, t_{i+1}]\in \Pi_1} \|F(t_{i+1}) - F(t_i)\|^p + \sum_{[t_i, t_{i+1}]\in \Pi_2} \|F(t_{i+1}) - F(t_i) \|^p \le \sup_{\Pi } \sum_{[t_i, t_{i+1}]\in \Pi} \|F(t_{i+1}) - F(t_i) \|^p = c(s, t).$$
“Taking suprema on the two terms on the left-hand side, we obtain
\[c(s, u) + c(u, t) \le c(s, t).\qedhere\]
\end{proof}

We next recall the notion of reduced rough paths. \vspace{0.1in}

\begin{definition}~\cite[Definition 4.6]{CP19}
Let $p \ge 1$. A \textit{reduced rough path of finite p-variation} is a tuple
\[
\x = (1, \x^1, \dots, \x^{\lfloor p \rfloor}) : \Delta_T \to \mathbb{S}_{\lfloor p \rfloor}(V),
\]
such that
\begin{enumerate}
\item[(i)] there exists a control function $c$ with
\[\sum_{k=1}^{\lfloor p \rfloor} \|\mathbb{X}_{s,t}^k\|^{p/k} \le c(s,t), \quad (s,t) \in \Delta_T;\]

\item[(ii)] the \textit{reduced Chen relation holds}
\[\mathbb{X}_{s,t} = \text{Sym}(\mathbb{X}_{s,u} \otimes \mathbb{X}_{u,t}), \quad (s,u), (u,t) \in \Delta_T, \]
where the symmetric part of $T \in T_k(V)$ is defined as
\[\text{Sym}(T) := \frac{1}{k!} \sum_{\sigma \in \mathfrak{S}_k} \sigma T, \quad \sigma T(v_1, \dots, v_k) := T(v_{\sigma (1)}, \dots, v_{\sigma (k)})\]
with the group of permutations $\mathfrak{S}_k$ of $\{1, \dots, k\}$.
\end{enumerate}
\mlabel{def:1}
\end{definition}

We can construct a reduced rough path solely from a finite $p$-variation path $X$. This is the jump-path analogue of~\cite[Lemma 4.7]{CP19}.
\begin{proposition}
Let $p \ge 1$. Let $X: [0, T]\to V$ and $\|X\|_{p\textup{-var}; [0, T]}<\infty$. Denote
\[\x^k:=\frac{1}{k!}X^{\otimes k},\quad \forall k=1, \ldots, \lfloor p \rfloor.\]
Then $\x = (1, \x^1, \dots, \x^{\lfloor p \rfloor})$ is a reduced rough path of finite p-variation.
\mlabel{prop:1}
\end{proposition}

\begin{proof}
According to Definition~\ref{def:1}, our first step is to prove that $\x^k$ has finite $p/k$-variation. It is known that $X$ has finite $p$-variation, i.e. $\|X\|_{p\text{-var}}<\infty$. By Lemma~\ref{lem:1}, there exists a control function $c(s, t)$ such that
\[\|X_{s, t}\|^p\le c(s, t), \quad \forall 0\le s\le t\le T. \]
Then
\begin{equation}
\|X_{s, t}\|\le c(s, t)^{1/p}.
\mlabel{eq:1}
\end{equation}
Consider the norm of the component $\x^k_{s,t}$ of the $k$-th layer. By the definition of $\x^k_{s,t}$, we have
\begin{equation}
\|\mathbb{X}^k_{s,t}\| = \left\| \frac{1}{k!} X_{s,t}^{\otimes k} \right\| = \frac{1}{k!} \|X_{s,t}\|^k\overset{(\ref{eq:1})}{\le}\frac{1}{k!} \left( c(s, t)^{1/p} \right)^k = \frac{1}{k!} c(s, t)^{k/p}.
\mlabel{eq:2}
\end{equation}
Raising both sides of~(\ref{eq:2}) to the power $p/k$ gives
$$\|\mathbb{X}^k_{s,t}\|^{p/k} \le \left( \frac{1}{k!} \right)^{p/k} c(s, t).$$
For any finite partition $\Pi = \{0 = t_0 < t_1 < \dots < t_n = T\}$ of the interval $[0, T]$, sum the $p/k$-th powers of its increments:
\begin{align}
\sum_{i=0}^{n-1} \|\mathbb{X}^k_{t_i, t_{i+1}}\|^{p/k}\le&\  \left( \frac{1}{k!} \right)^{p/k} \sum_{i=0}^{n-1} c(t_i, t_{i+1})\nonumber \\
\le&\ \left( \frac{1}{k!} \right)^{p/k}c(0, T) \hspace{2cm}(\text{by Definition~\ref{def:2}}). \mlabel{eq:3}
\end{align}
Since this inequality holds for any partition of $[0, T]$, and based on the definition of $p$-variation, we obtain
\begin{align*}
\|\mathbb{X}^k\|_{p/k\text{-var}; [0, T]}=&\ \Big(\sup_{\Pi } \sum_{[t_i, t_{i+1}] \in \Pi} \|\mathbb{X}^k_{t_i, t_{i+1}}\|^{p/k}\Big)^{k/p} \\
\le&\ \frac{1}{k!}c(0, T)^{k/p} \hspace{2cm}(\text{by~(\ref{eq:3})})\\
\le&\ \max\{1, c(0, T)\}\hspace{2cm}(\text{by $\frac{1}{k!}\le 1$ and $k/p<1$})\\
<&\ \infty.
\end{align*}

The proof of the reduced Chen relation proceeds along the lines of Lemma 4.7 in~\cite{CP19}.
\end{proof}

We next introduce reduced controlled rough paths with jumps. We keep the definition from~\cite{CP19}, except that we do not require the reduced controlled rough path $\Y$ to be continuous.
\begin{definition}
Let $p \ge 1$ and $\mathbb{X}$ be a reduced rough path of finite $p$-variation. A path
\[
\Y = (Y^0, Y^1, \dots, Y^{\lfloor p \rfloor-1}):[0, T]\to \Big(W, \mathcal{L}\big(\text{Sym}_1(V), W\big),\ldots, \mathcal{L}\big(\text{Sym}_{\lfloor p \rfloor-1}(V), W\big)   \Big)
\]
is a reduced $\X$-controlled rough path if
\[\|R^i\|_{\frac{p}{\lfloor p \rfloor - i }-var}<\infty,\quad \text{for\ } i=0, \dots, \lfloor p \rfloor-1,  \]
where
\begin{equation*}
R^i_{s, t}:=
\begin{cases}
 Y_t^i-Y_s^i-\sum_{j=1}^{\lfloor p \rfloor-1-i}Y^{i+j}_s\x^j_{s, t},\quad i=0, \ldots, \lfloor p \rfloor-2;\\
Y_t^{\lfloor p \rfloor-1}-Y_s^{\lfloor p \rfloor-1},\hspace{2.42cm} i=\lfloor p \rfloor-1.
\end{cases}
\end{equation*}
In that case we write $\Y \in \mathcal{D}^{p}_{\mathbb{X}}([0,T], W)$.
\end{definition}

We end this subsection by recording a special reduced controlled rough path. For $k\in \ZZ_{\geq 1}$, denote by
\begin{align*}
\mathcal{C}^k_b(W, U):= \{ F:W \to U \text{ is } k \text{ times continuously differentiable
  and satisfies } \|D^jF\|_{\infty} < \infty, j=0, \ldots, k\}.
\end{align*}
Here $D^jF: W\to\mathcal L\big((W)^{\otimes j}, U\big)$ denotes the $j$-th differential of $F$.
\begin{lemma}
Let $p\ge 1$ and $\x$ be as in Proposition~\ref{prop:1}. Let $F\in \mathcal{C}^{\lfloor p \rfloor+1}_b(V, U)$, then
\[\Big(DF(X), D^2F(X), \dots, D^{\lfloor p \rfloor}F(X)\Big)\in \mathcal{D}^{p}_{\mathbb{X}}\Big([0,T], L(V, U)\Big). \]
\end{lemma}

\begin{proof}
For $i=0, \ldots, \lfloor p \rfloor-2$,
\begin{align*}
\|R^i_{s, t}\|=&\ \Big\| D^{i+1}F(X_t)-D^{i+1}F(X_s)-\sum_{j=1}^{\lfloor p \rfloor-1-i}D^{i+1+j}F(X_s)\x^j_{s, t}\Big\| \\
=&\ \Big\|D^{i+1}F(X_t)-D^{i+1}F(X_s)-\sum_{j=1}^{\lfloor p \rfloor-1-i} \frac{1}{j!}D^{i+1+j}F(X_s)X^{\otimes j}_{s, t}\Big\|\\
=&\ O\Big(c(s, t)^{(\lfloor p \rfloor-i)/p}\Big) \hspace{2cm}(\text{by Taylor's formula}).
\end{align*}

For $i=\lfloor p \rfloor-1$,
\begin{align*}
\Big\|R^{\lfloor p \rfloor-1}_{s, t}\Big\|=&\ \Big\|D^{\lfloor p \rfloor}F(X_t)-D^{\lfloor p \rfloor}F(X_s)\Big\| \\
\le&\ \|D^{\lfloor p \rfloor+1}F\|_{\infty}\|X_t-X_s\| \\
=&\ O\Big(c(s, t)^{1/p}\Big) \hspace{1cm}(\text{by Definition~\ref{def:1}~(i)}).
\end{align*}

Hence
\[\Big(DF(X), D^2F(X), \dots, D^{\lfloor p \rfloor}F(X)\Big)\in \mathcal{D}^{p}_{\mathbb{X}}\Big([0,T], L(V, U)\Big). \qedhere\]
\end{proof}

\vspace{0.4cm}

\subsection{Reduced rough integral with jumps}
The classical convergence criteria for integration no longer apply to jump paths. We review a convergence criterion for integration that is applicable to jump paths.

\begin{definition}~\cite{Hi38, Hi63}
For a partition $\Pi$ of $[0, T]$ and any $[u, v] \in \Pi, \Xi_{u,v}$ takes values in $\mathbb{R}^d$. 
We say that the Riemann sum $\sum_{[u,v]\in \Pi} \Xi_{u,v}$ converges to $K$ in the Refinement Riemann-Stieltjes (RRS) sense if for any $\varepsilon  > 0$, there exists $\Pi_\varepsilon$ such that for any refinement $\Pi \supset \Pi_\varepsilon$, one has $|\sum_{[u,v]\in \Pi} \Xi_{u,v} - K| < \varepsilon$.
\end{definition}

We now give the sewing lemma necessary for rough integration with jumps.

\begin{lemma}~\cite[Theorem 2.5]{FZ18}
Suppose $\Xi:\Delta_T \to V$. Define
\[\delta \Xi_{s,u,t}:=\Xi_{s, t}-\Xi_{s, u}-\Xi_{u, t}.\]
Assume $\delta \Xi$ satisfies
\[\delta \Xi_{s,u,t} =O\Big(c_{1}^{\alpha_{1}}(s,u)c_{2}^{\alpha_{2}}(u,t) \Big) ,\]
where $c_{1}, c_{2}$ are controls and $\alpha_{1} + \alpha_{2} > 1$. Then the following limit exists and is unique in the RRS sense:
\[ \text{RRS} - \lim_{\Pi} \sum_{[u,v]\in\Pi}\Xi_{u,v}, \]
where the limit is taken over refinements of partitions $\Pi$ of $[0, T]$.
\mlabel{lem:3}
\end{lemma}

\begin{proposition}
Let $p \ge 1$, let $\x = (1, \x^1, \dots, \x^{\lfloor p \rfloor})$ be a reduced rough path of finite $p$-variation and let $\Y= (Y^0, Y^1, \dots, Y^{\lfloor p \rfloor-1}) \in \mathcal{D}^{p}_{\mathbb{X}}\Big([0,T], L(V, U)\Big)$. Then the reduced rough integral
\[\int_0^tY^0_rd\,\X_r:=\text{RRS} - \lim_{\Pi} \sum_{[t_j, t_{j+1}] \in \Pi} \sum_{k=0}^{\lfloor p \rfloor-1} Y^k_{t_j}\x^{k+1}_{t_j, t_{j+1}} , \quad t \in [0,T]  \]
is well defined, where the limit is taken over refinements of partitions $\Pi$ of $[0, T]$.
\mlabel{prop:2}
\end{proposition}

\begin{proof}
Let $\Xi_{s, t}:=\sum_{k=0}^{\lfloor p \rfloor-1}Y^k_s\mathbb{X}^{k+1}_{s,t}$. We need to show that
\[\delta \Xi_{s,u,t}=\sum_{k=0}^{\lfloor p \rfloor-1}Y^k_s\mathbb{X}^{k+1}_{s,t}-\sum_{k=0}^{\lfloor p \rfloor-1}Y^k_s\mathbb{X}^{k+1}_{s,u}-\sum_{k=0}^{\lfloor p \rfloor-1}Y^k_u\mathbb{X}^{k+1}_{u,t} \]
satisfies the conditions specified in Lemma~\ref{lem:3}. Since $\Y \in \mathcal{D}_{\mathbb{X}}^{p}$ is a reduced controlled path, we expand $Y_u^k$ at $s$:
\begin{align*}
 \sum_{k=0}^{\lfloor p \rfloor-1} Y^k_u \mathbb{X}^{k+1}_{u,t} 
&= \sum_{k=0}^{\lfloor p \rfloor-1} \Bigg( \sum_{j=k}^{\lfloor p \rfloor-1} Y^j_s\mathbb{X}^{j-k}_{s,u} + R^k_{s,u} \Bigg) \otimes \mathbb{X}^{k+1}_{u,t} \\
&= \sum_{j=0}^{\lfloor p \rfloor-1} Y^j_s \left( \sum_{k=0}^j \mathbb{X}^{j-k}_{s,u} \otimes \mathbb{X}^{k+1}_{u,t} \right) + \sum_{k=0}^{\lfloor p \rfloor-1} R^k_{s,u} \otimes \mathbb{X}^{k+1}_{u,t} \\
&= \sum_{k=0}^{\lfloor p \rfloor-1} Y^k_s \left( \mathbb{X}^{k+1}_{s,t} - \mathbb{X}^{k+1}_{s,u} \right) + \sum_{k=0}^{\lfloor p \rfloor-1} R^k_{s,u} \otimes \mathbb{X}^{k+1}_{u,t},
\end{align*}
where in the second step we reordered the double sum, and in the last step we used the reduced Chen relation for the $(k+1)$-th level: $\mathbb{X}^{k+1}_{s,t} = \text{Sym} \sum_{j=0}^{k+1} \mathbb{X}^{k+1-j}_{s,u} \otimes \mathbb{X}^j_{u,t}$, noting that $Y^k_s$ is symmetric. Therefore,
\begin{align*}
\delta \Xi_{s,u,t} 
&= \sum_{k=0}^{\lfloor p \rfloor-1} Y^k_s \mathbb{X}^{k+1}_{s,t} - \sum_{k=0}^{\lfloor p \rfloor-1} Y^k_s \mathbb{X}^{k+1}_{s,u} - \left( \sum_{k=0}^{\lfloor p \rfloor-1} Y^k_s (\mathbb{X}^{k+1}_{s,t} - \mathbb{X}^{k+1}_{s,u}) + \sum_{k=0}^{\lfloor p \rfloor-1} R^k_{s,u} \otimes \mathbb{X}^{k+1}_{u,t} \right) \\
&= - \sum_{k=0}^{\lfloor p \rfloor-1} R^k_{s,u} \otimes \mathbb{X}^{k+1}_{u,t}.
\end{align*}
According to Definition~\ref{def:1} and Definition~\ref{def:2}, we have $\|R^k_{s,u}\| = O(c(s,u)^{\frac{\lfloor p \rfloor - k}{p}})$ and $\|\mathbb{X}^{k+1}_{u,t}\| = O(c(u,t)^{\frac{k+1}{p}})$. Thus,
\[ \|\delta \Xi_{s,u,t}\| = \sum_{k=0}^{\lfloor p \rfloor-1} O(c(s,u)^{\frac{\lfloor p \rfloor - k}{p}} c(u,t)^{\frac{k+1}{p}}). \]
Since the sum of exponents $\frac{\lfloor p \rfloor - k}{p} + \frac{k+1}{p} = \frac{\lfloor p \rfloor + 1}{p} > 1$ for all $k$, the result follows from Lemma~\ref{lem:3}.
\end{proof}

\section{The jump It\^o-type formula}\label{sec:3}
In this section, we present the It\^o-type formula for jump paths with arbitrary regularity. Before proving the main formula, we recall the following lemma as a preparation.


\begin{lemma}~\cite[Lemma 1.6]{FZ18}
Let $p \ge 1$. Let $X: [0, T]\to V$ and $\|X\|_{p\text{-var}; [0, T]}<\infty$. Then for any $\varepsilon > 0$, there exists a partition $\Pi$ such that for any interval $[s, t] \in \Pi$,
$\sup_{u, v \in (s, t)} \|X_{u, v}\| < \varepsilon.$
\mlabel{lem:4}
\end{lemma}

The following notations will be employed in Theorem~\ref{thm:main}:

\begin{align*}
X_{t-}:=&\ \lim_{u\uparrow t}X_u,\quad X_{t+}:=\ \lim_{u\downarrow t}X_u, \quad \Delta_t^- X:= X_{t-, t}=\lim_{u\uparrow t}X_{u, t}, \\
 \quad \Delta_t^+ X:=&\ \ X_{t, t+}=\lim_{u\downarrow t}X_{t, u}, \quad X_{s+, t-}:= \lim_{v\downarrow s,\ u\uparrow t}X_{v, u}.
\end{align*}
We now state the jump It\^o-type formula.

\begin{theorem}
Let $p \ge 1$, and let $X:  [0, T]\to V$ and $\|X\|_{p\text{-var}; [0, T]}<\infty$. Suppose $F\in\mathcal{C}^{\lfloor p \rfloor+1}_b(V, U)$.
Then  
\begin{enumerate}

\item 
\begin{equation}
\begin{aligned}
F(X_T) - F(X_0) =&\ \int_0^T D F(X_t) d\,\X_t+\sum_{0<t\le T}\Big( F(X_t) - F(X_{t-}) - \sum_{k=1}^{\lfloor p \rfloor} \frac{1}{k!} D^k F(X_{t-})(\Delta_t^- X)^{\otimes k}\Big)  \\
&\ +\sum_{0\le t< T}\Big(F(X_{t+}) - F(X_t) - \sum_{k=1}^{\lfloor p \rfloor} \frac{1}{k!} D^k F(X_{t})(\Delta_t^+ X)^{\otimes k}\Big),
\end{aligned}
\mlabel{eq:maina}
\end{equation}
where
\[\int_0^T D F(X_t) d\,\X_t:=\text{RRS} - \lim_{\Pi} \sum_{[t_i, t_{i+1}] \in \Pi}\sum_{k=1}^{\left \lfloor p \right \rfloor }\frac{1}{k!}D^kF(X_{t_{i}})X_{t_{i}, t_{i+1}}^{\otimes k}.\]\mlabel{it:bitem1}

\item If X is a c\`adl\`ag path (right-continuous), then~\meqref{eq:maina} reduces to
\begin{align*}
F(X_T) - F(X_0) =&\ \int_0^T D F(X_t) d\,\X_t+\sum_{0<t\le T}\Big( F(X_t) - F(X_{t-}) - \sum_{k=1}^{\lfloor p \rfloor} \frac{1}{k!} D^k F(X_{t-})(\Delta_t^- X)^{\otimes k}\Big).
\end{align*}\mlabel{it:bitem2}

\item If X is a continuous path, then~\meqref{eq:maina} reduces to
\begin{equation*}
F(X_T) - F(X_0) =\int_0^T D F(X_t) d\,\X_t.
\end{equation*}
This is the continuous pathwise change-of-variable formula given in~\cite{CP19}.\mlabel{it:bitem3}
\end{enumerate}
\mlabel{thm:main}
\end{theorem}

\vspace{0.5cm}

\begin{proof}
It suffices to prove~\eqref{it:bitem1}, as~\eqref{it:bitem2} and~\eqref{it:bitem3} can be directly derived from~\eqref{it:bitem1}.
Since $X$ has finite $p$-variation on $[0,T]$, it is a regulated path, ensuring that the left and right limits $X_{t-}$ and $X_{t+}$ exist for all $t \in [0,T]$.
For any partition $\Pi=\{0=t_0<t_1<\cdots<t_N=T\}$ of $[0, T]$, we have
\begin{align*}
F(X_T) - F(X_0)=&\ \sum_{i=0}^{N-1}\bigg(F(X_{t_{i+1}}) - F(X_{t_{i}})-\sum_{k=1}^{\left \lfloor p \right \rfloor }\frac{1}{k!}D^kF(X_{t_{i}})X_{t_{i}, t_{i+1}}^{\otimes k} \bigg)+ \sum_{i=0}^{N-1}\sum_{k=1}^{\left \lfloor p \right \rfloor }\frac{1}{k!}D^kF(X_{t_{i}})X_{t_{i}, t_{i+1}}^{\otimes k} \\
=&\ \sum_{i=0}^{N-1}\Big(\left(F(X_{t_{i+1}})-F(X_{t_{i+1-}}) \right)+\left(F(X_{t_{i+}})-F(X_{t_{i}}) \right)\Big)\\
&+\sum_{i=0}^{N-1}\bigg(F(X_{t_{i+1}-}) - F(X_{t_{i}+})-\sum_{k=1}^{\left \lfloor p \right \rfloor }\frac{1}{k!}D^kF(X_{t_{i}})X_{t_{i}, t_{i+1}}^{\otimes k} \bigg)+ \sum_{i=0}^{N-1}\sum_{k=1}^{\left \lfloor p \right \rfloor }\frac{1}{k!}D^kF(X_{t_{i}})X_{t_{i}, t_{i+1}}^{\otimes k}.
\end{align*}
Denote
\begin{align*}
A_1:=&\ \sum_{i=0}^{N-1}\sum_{k=1}^{\left \lfloor p \right \rfloor }\frac{1}{k!}D^kF(X_{t_{i}})X_{t_{i}, t_{i+1}}^{\otimes k}; \\
A_2:=&\ \sum_{i=0}^{N-1}\Big(\left(F(X_{t_{i+1}})-F(X_{t_{i+1-}}) \right)+\left(F(X_{t_{i+}})-F(X_{t_{i}}) \right)\Big).
\end{align*}
Then
\begin{align*}
F(X_T) - F(X_0)=&\ A_2+\sum_{i=0}^{N-1}\bigg(F(X_{t_{i+1}-}) - F(X_{t_{i}+})-\sum_{k=1}^{\left \lfloor p \right \rfloor }\frac{1}{k!}D^kF(X_{t_{i}})X_{t_{i}, t_{i+1}}^{\otimes k} \bigg)+A_1 \\
=&\ A_2+\sum_{i=0}^{N-1}\bigg(F(X_{t_{i+1}-}) - F(X_{t_{i}+})-\sum_{k=1}^{\left \lfloor p \right \rfloor }\frac{1}{k!}D^kF(X_{t_{i}+})X_{t_{i}+, t_{i+1}-}^{\otimes k} \bigg)\\
&+\sum_{i=0}^{N-1}\sum_{k=1}^{\left \lfloor p \right \rfloor }\bigg(\frac{1}{k!}D^kF(X_{t_{i}+})X_{t_{i}+, t_{i+1}-}^{\otimes k}-\frac{1}{k!}D^kF(X_{t_{i}})X_{t_{i}, t_{i+1}}^{\otimes k}\bigg)+A_1.
\end{align*}
Let
\begin{align*}
B_1:=&\ \sum_{i=0}^{N-1}\bigg(F(X_{t_{i+1}-}) - F(X_{t_{i}+})-\sum_{k=1}^{\left \lfloor p \right \rfloor }\frac{1}{k!}D^kF(X_{t_{i}+})X_{t_{i}+, t_{i+1}-}^{\otimes k} \bigg).
\end{align*}
We have
\begin{align*}
&\ F(X_T) - F(X_0) \\
=&\ A_2+B_1+\sum_{i=0}^{N-1}\sum_{k=1}^{\left \lfloor p \right \rfloor }\bigg(\frac{1}{k!}D^kF(X_{t_{i}+})X_{t_{i}+, t_{i+1}-}^{\otimes k}-\frac{1}{k!}D^kF(X_{t_{i}})X_{t_{i}, t_{i+1}}^{\otimes k}\bigg)+A_1 \\
=&\ A_2+B_1+\sum_{i=0}^{N-1}\bigg(DF(X_{t_{i}+})-DF(X_{t_{i}})-\sum_{k=2}^{\left \lfloor p \right \rfloor }\frac{1}{(k-1)!}D^kF(X_{t_{i}})(\Delta_{t_i}^+X)^{\otimes k-1} \bigg)X_{t_{i}+, t_{i+1}-}\\
&+\sum_{i=0}^{N-1}\bigg(DF(X_{t_{i}})X_{t_{i}+, t_{i+1}-}-DF(X_{t_{i}})X_{t_{i}, t_{i+1}}\bigg)\\
&+\sum_{i=0}^{N-1}\sum_{k=2}^{\left \lfloor p \right \rfloor }\bigg(\frac{1}{k!}D^kF(X_{t_{i}+})X_{t_{i}+, t_{i+1}-}^{\otimes k}+\frac{1}{(k-1)!}D^kF(X_{t_{i}})(\Delta_{t_i}^+X)^{\otimes k-1}X_{t_{i}+, t_{i+1}-}-\frac{1}{k!}D^kF(X_{t_{i}})X_{t_{i}, t_{i+1}}^{\otimes k}\bigg)+A_1.
\end{align*}
Set
\begin{align*}
B_2:=&\ \sum_{i=0}^{N-1}\bigg(DF(X_{t_{i}+})-DF(X_{t_{i}})-\sum_{k=2}^{\left \lfloor p \right \rfloor }\frac{1}{(k-1)!}D^kF(X_{t_{i}})(\Delta_{t_i}^+X)^{\otimes k-1} \bigg)X_{t_{i}+, t_{i+1}-}.
\end{align*}
We obtain
\begin{align*}
&\ F(X_T) - F(X_0)\\
=&\  A_2+B_1+B_2+\sum_{i=0}^{N-1}DF(X_{t_{i}})\bigg(-\Delta_{t_i}^+X-\Delta_{t_{i+1}}^-X\bigg)+\sum_{i=0}^{N-1}\sum_{k=2}^{\left \lfloor p \right \rfloor }\bigg(\frac{1}{k!}D^kF(X_{t_{i}+})X_{t_{i}+, t_{i+1}-}^{\otimes k}\\
&+\frac{1}{(k-1)!}D^kF(X_{t_{i}})(\Delta_{t_i}^+X)^{\otimes k-1}X_{t_{i}+, t_{i+1}-}-\frac{1}{k!}D^kF(X_{t_{i}})\Big(\Delta_{t_i}^+X+X_{t_{i}+, t_{i+1}-}+\Delta_{t_{i+1}}^-X\Big)^{\otimes k}\bigg)+A_1\\
=&\ A_2+B_1+B_2+\sum_{i=0}^{N-1}\sum_{k=1}^{\left \lfloor p \right \rfloor }-\frac{1}{k!}D^kF(X_{t_{i}})(\Delta_{t_i}^+X)^{\otimes k}+\sum_{i=0}^{N-1}\sum_{k=1}^{\left \lfloor p \right \rfloor }-\frac{1}{k!}D^kF(X_{t_{i+1}-})(\Delta_{t_{i+1}}^-X)^{\otimes k}+B_3+A_1\\
=&\ A_2+B_1+B_2+A_3+B_3+A_1,
\end{align*}
where
\begin{align}
A_3:=&\ \sum_{i=0}^{N-1}\sum_{k=1}^{\left \lfloor p \right \rfloor }-\frac{1}{k!}D^kF(X_{t_{i}})(\Delta_{t_i}^+X)^{\otimes k}+\sum_{i=0}^{N-1}\sum_{k=1}^{\left \lfloor p \right \rfloor }-\frac{1}{k!}D^kF(X_{t_{i+1}-})(\Delta_{t_{i+1}}^-X)^{\otimes k}\nonumber\\
B_3:=&\ \sum_{i=0}^{N-1}\sum_{k=1}^{\left \lfloor p \right \rfloor }\frac{1}{k!}\Big(D^kF(X_{t_{i+1}-})-D^kF(X_{t_{i}})\Big)(\Delta_{t_{i+1}}^-X)^{\otimes k}+\sum_{i=0}^{N-1}\sum_{k=2}^{\left \lfloor p \right \rfloor }\frac{1}{k!}\Big(D^kF(X_{t_{i}+})-D^kF(X_{t_{i}}) \Big)X_{t_{i}+, t_{i+1}-}^{\otimes k}\nonumber \\
&-\sum_{i=0}^{N-1}\sum_{k=3}^{\left \lfloor p \right \rfloor }\frac{1}{k!}D^kF(X_{t_{i}})\sum_{m=1}^{k-2} C_k^m(\Delta_{t_i}^+X)^{\otimes m}X_{t_{i}+, t_{i+1}-}^{\otimes k-m}-\sum_{i=0}^{N-1}\sum_{k=2}^{\left \lfloor p \right \rfloor }\frac{1}{k!}D^kF(X_{t_{i}})\sum_{m=1}^{k-1}C_k^m(\Delta_{t_{i+1}}^-X)^{\otimes m}X_{t_{i}+, t_{i+1}-}^{\otimes k-m} \nonumber\\
&-\sum_{i=0}^{N-1}\sum_{k=2}^{\left \lfloor p \right \rfloor }\frac{1}{k!}D^kF(X_{t_{i}})\sum_{m=1}^{k-1}C_k^m(\Delta_{t_i}^+X)^{\otimes m}(\Delta_{t_{i+1}}^-X)^{\otimes k-m}.\mlabel{eq:ad0}
\end{align}

We estimate these terms separately. Note that the set of jump points of a finite $p$-variation path is at most countable. By Lemma~\ref{lem:4}, we can choose a partition $\Pi=\{0=t_0<t_1<\cdots<t_N=T\}$ such that for any $[t_i, t_{i+1}]\in \Pi$, $X$ is continuous at either $t_i$ or $t_{i+1}$, and furthermore
$$\sup_{u, v \in (t_i, t_{i+1})} \|X_{u, v}\| < \varepsilon.$$
We divide the remaining proof into two steps. 

\noindent{\bf Step 1. } We prove that $B_1$, $B_2$ and $B_3$ converge to zero along such partitions.
For $B_1$,
\begin{align*}
\|B_1\|=&\ \bigg\|\sum_{i=0}^{N-1}\bigg(F(X_{t_{i+1}-}) - F(X_{t_{i}+})-\sum_{k=1}^{\left \lfloor p \right \rfloor }\frac{1}{k!}D^kF(X_{t_{i}+})X_{t_{i}+, t_{i+1}-}^{\otimes k} \bigg)\bigg\|\\
\le&\ \sum_{i=0}^{N-1}\bigg\|F(X_{t_{i+1}-}) - F(X_{t_{i}+})-\sum_{k=1}^{\left \lfloor p \right \rfloor }\frac{1}{k!}D^kF(X_{t_{i}+})X_{t_{i}+, t_{i+1}-}^{\otimes k}\bigg\|\\
\le&\ \sum_{i=0}^{N-1}C\Big\|X_{t_{i}+, t_{i+1}-}^{\otimes \left \lfloor p \right \rfloor+1}\Big\|\hspace{1cm}(\text{by the Taylor expression of $F$ at $X_{t_{i}+}$})\\
=&\ \sum_{i=0}^{N-1}C\Big\|X_{t_{i}+, t_{i+1}-}\Big\|^{ \left \lfloor p \right \rfloor+1-p}\,\Big\|X_{t_{i}+, t_{i+1}-}\Big\|^{p}\\
<&\ \sum_{i=0}^{N-1}C\varepsilon^{ \left \lfloor p \right \rfloor+1-p}\,\Big\|X_{t_{i}+, t_{i+1}-}\Big\|^{p}\\
=&\ C\varepsilon^{ \left \lfloor p \right \rfloor+1-p}\sum_{i=0}^{N-1}\Big\|X_{t_{i}+, t_{i+1}-}\Big\|^{p}\\
\le&\ C\varepsilon^{ \left \lfloor p \right \rfloor+1-p} \hspace{1cm}\Big(\text{by $\|X\|_{p\text{-var};[0 ,T]}<\infty$ and $C:=C\|X\|_{p\text{-var};[0 ,T]}^{p}$  }\Big).
\end{align*}

For $B_2$,
\begin{align*}
\|B_2\|=&\ \bigg\|\sum_{i=0}^{N-1}\bigg(DF(X_{t_{i}+})-DF(X_{t_{i}})-\sum_{k=2}^{\left \lfloor p \right \rfloor }\frac{1}{(k-1)!}D^kF(X_{t_{i}})(\Delta_{t_i}^+X)^{\otimes k-1} \bigg)X_{t_{i}+, t_{i+1}-}\bigg\|\\
\le&\ \sum_{i=0}^{N-1}\bigg\|\bigg(DF(X_{t_i} + \Delta_{t_i}^+ X)-DF(X_{t_{i}})-\sum_{k=2}^{\left \lfloor p \right \rfloor }\frac{1}{(k-1)!}D^kF(X_{t_{i}})(\Delta_{t_i}^+X)^{\otimes k-1} \bigg)X_{t_{i}+, t_{i+1}-}\bigg\|\\
\le&\ \sum_{i=0}^{N-1}C\Big\|(\Delta_{t_i}^+X)^{\otimes \left \lfloor p \right \rfloor}\Big\|\,\Big\|X_{t_{i}+, t_{i+1}-}\Big\|\hspace{1cm}(\text{by the Taylor expression of $DF$ at $X_{t_{i}}$}).
\end{align*}
Since $\|X\|_{p\text{-var}; [0, T]}<\infty$, for $t\in [0, T]$ and any fixed threshold $\eta > 0$, there are only finitely many points satisfying $\|\Delta_{t}^+X\|>\eta$. Using this property, we estimate the last display as follows.

\begin{align*}
\|B_2\| \le&\ \sum_{\|\Delta_{t_i}^+X\| > \eta} C \Big\|\Delta_{t_i}^+X\Big\|^{\lfloor p \rfloor} \Big\|X_{t_{i}+, t_{i+1}-}\Big\| + \sum_{\|\Delta_{t_i}^+X\| \le \eta} C \Big\|\Delta_{t_i}^+X\Big\|^{\lfloor p \rfloor} \Big\|X_{t_{i}+, t_{i+1}-}\Big\| \\
\le&\ \max_{i} \Big\|X_{t_{i}+, t_{i+1}-}\Big\| \cdot \bigg( C \sum_{\|\Delta_{t_i}^+X\| > \eta} \Big\|\Delta_{t_i}^+X\Big\|^{\lfloor p \rfloor} \bigg) + C \sum_{\|\Delta_{t_i}^+X\| \le \eta} \Big\|\Delta_{t_i}^+X\Big\|^p \Big\|\Delta_{t_i}^+X\Big\|^{\lfloor p \rfloor - p} \Big\|X_{t_{i}+, t_{i+1}-}\Big\| \\
\le&\ \varepsilon \cdot C_\eta + C \eta^{\lfloor p \rfloor + 1 - p} \sum_{\|\Delta_{t_i}^+X\| \le \eta} \Big\|\Delta_{t_i}^+X\Big\|^p \hspace{1cm}\bigg(\text{by setting $C_\eta:=C \sum_{\|\Delta_{t_i}^+X\| > \eta} \Big\|\Delta_{t_i}^+X\Big\|^{\lfloor p \rfloor}$}\bigg)\\
\le&\ \varepsilon \cdot C_\eta + C \eta^{\lfloor p \rfloor + 1 - p} \|X\|_{p\text{-var}; [0, T]}^p.
\end{align*}
Since $\lfloor p \rfloor + 1 > p$, taking the RRS limit as $\varepsilon \to 0$ and then letting the threshold $\eta \to 0$, both terms converge to zero. 

For $B_3$, since $X$ is continuous at either $t_i$ or $t_{i+1}$, we have
\begin{equation}
\Delta_{t_i}^+X\Delta_{t_{i+1}}^-X=0.
\mlabel{eq:ad1}
\end{equation}
First, we handle the first and second summations on the right-hand side of~(\ref{eq:ad0}).
\begin{align}
&\ \sum_{i=0}^{N-1}\sum_{k=1}^{\left \lfloor p \right \rfloor }\frac{1}{k!}\Big(D^kF(X_{t_{i+1}-})-D^kF(X_{t_{i}})\Big)(\Delta_{t_{i+1}}^-X)^{\otimes k} \nonumber\\
=&\ \sum_{i=0}^{N-1} \sum_{k=1}^{\left \lfloor p \right \rfloor } \frac{1}{k!} \bigg(\sum_{m=1}^{\left \lfloor p \right \rfloor -k} \frac{1}{m!} D^{k+m} F(X_{t_i}) (X_{t_i, t_{i+1}-})^{\otimes m}+O\Big(X_{t_i, t_{i+1}-}^{\left \lfloor p \right \rfloor -k+1}\Big)\bigg)(\Delta_{t_{i+1}}^- X)^{\otimes k} \nonumber\\
&\ \hspace{7cm}(\text{by the Taylor expression of $D^kF$ at $X_{t_{i}}$})\nonumber\\
=&\ \sum_{i=0}^{N-1} \sum_{k=1}^{\left \lfloor p \right \rfloor } \frac{1}{k!} \bigg(\sum_{m=1}^{\left \lfloor p \right \rfloor -k} \frac{1}{m!} D^{k+m} F(X_{t_i}) \big(\Delta_{t_i}^+X+X_{t_i+, t_{i+1}-}\big)^{\otimes m}+O\Big(\big(\Delta_{t_i}^+X+X_{t_i+, t_{i+1}-}\big)^{\left \lfloor p \right \rfloor -k+1}\Big)\bigg)(\Delta_{t_{i+1}}^- X)^{\otimes k} \nonumber\\
=&\ \sum_{i=0}^{N-1} \sum_{k=1}^{\left \lfloor p \right \rfloor } \frac{1}{k!} \bigg(\sum_{m=1}^{\left \lfloor p \right \rfloor -k} \frac{1}{m!} D^{k+m} F(X_{t_i}) X_{t_i+, t_{i+1}-}^{\otimes m}+O\Big(X_{t_i+, t_{i+1}-}^{\left \lfloor p \right \rfloor -k+1}\Big)\bigg)(\Delta_{t_{i+1}}^- X)^{\otimes k} \hspace{2cm}(\text{by~(\ref{eq:ad1})})\nonumber\\
=&\ \sum_{i=0}^{N-1} \sum_{k=1}^{\left \lfloor p \right \rfloor } \frac{1}{k!} \sum_{m=1}^{\left \lfloor p \right \rfloor -k} \frac{1}{m!} D^{k+m} F(X_{t_i}) X_{t_i+, t_{i+1}-}^{\otimes m}(\Delta_{t_{i+1}}^- X)^{\otimes k}
+\sum_{i=0}^{N-1} \sum_{k=1}^{\left \lfloor p \right \rfloor } \frac{1}{k!}O\Big(X_{t_i+, t_{i+1}-}^{\left \lfloor p \right \rfloor -k+1}\Big)(\Delta_{t_{i+1}}^- X)^{\otimes k}\nonumber\\
=&\ \sum_{i=0}^{N-1}\sum_{k=2}^{\left \lfloor p \right \rfloor }\frac{1}{k!}D^kF(X_{t_{i}})\sum_{m=1}^{k-1}C_k^m(\Delta_{t_{i+1}}^-X)^{\otimes m}X_{t_{i}+, t_{i+1}-}^{\otimes k-m}+\sum_{i=0}^{N-1} \sum_{k=1}^{\left \lfloor p \right \rfloor } \frac{1}{k!}O\Big(X_{t_i+, t_{i+1}-}^{\left \lfloor p \right \rfloor -k+1}\Big)(\Delta_{t_{i+1}}^- X)^{\otimes k}.\mlabel{eq:ad2}
\end{align}
Similarly,  
\begin{align}
&\ \sum_{i=0}^{N-1}\sum_{k=2}^{\left \lfloor p \right \rfloor }\frac{1}{k!}\Big(D^kF(X_{t_{i}+})-D^kF(X_{t_{i}}) \Big)X_{t_{i}+, t_{i+1}-}^{\otimes k}\nonumber\\
=&\ \sum_{i=0}^{N-1}\sum_{k=3}^{\left \lfloor p \right \rfloor }\frac{1}{k!}D^kF(X_{t_{i}})\sum_{m=1}^{k-2} C_k^m(\Delta_{t_i}^+X)^{\otimes m}X_{t_{i}+, t_{i+1}-}^{\otimes k-m}+\sum_{i=0}^{N-1} \sum_{k=2}^{\left \lfloor p \right \rfloor } \frac{1}{k!}O\Big(X_{t_i+, t_{i+1}-}^{\left \lfloor p \right \rfloor -k+1}\Big)(\Delta_{t_{i}}^+ X)^{\otimes k}.\mlabel{eq:ad3}
\end{align}

Substituting~(\ref{eq:ad1}),~(\ref{eq:ad2}) and~(\ref{eq:ad3}) into~(\ref{eq:ad0}), 
\[B_3=\sum_{i=0}^{N-1} \sum_{k=1}^{\left \lfloor p \right \rfloor } \frac{1}{k!}O\Big(X_{t_i+, t_{i+1}-}^{\left \lfloor p \right \rfloor -k+1}\Big)(\Delta_{t_{i+1}}^- X)^{\otimes k}+\sum_{i=0}^{N-1} \sum_{k=2}^{\left \lfloor p \right \rfloor } \frac{1}{k!}O\Big(X_{t_i+, t_{i+1}-}^{\left \lfloor p \right \rfloor -k+1}\Big)(\Delta_{t_{i}}^+ X)^{\otimes k}.\]
Since $\left \lfloor p \right \rfloor-k+1+k=\left \lfloor p \right \rfloor+1>p$, the same argument as for $B_2$ shows that $B_3$ converges to zero.

\noindent{\bf Step 2. } We deal with $A_1$, $A_2$ and $A_3$. For $A_1$, note that $$\X=\Big(X, \frac{1}{2}X^{\otimes 2}, \ldots, \frac{1}{\left \lfloor p \right \rfloor!}X^{\otimes \left \lfloor p \right \rfloor}\Big)$$ is a reduced rough path and $\Big(DF(X), D^2F(X), \ldots, D^{\left \lfloor p \right \rfloor}F(X)\Big)$ is a reduced $\X$-controlled rough path. Then
\[\text{RRS} - \lim_{\Pi} A_1=\text{RRS} - \lim_{\Pi}\sum_{i=0}^{N-1}\sum_{k=1}^{\left \lfloor p \right \rfloor }\frac{1}{k!}D^kF(X_{t_{i}})X_{t_{i}, t_{i+1}}^{\otimes k}=\int_0^T D F(X_t) d\,\X_t.\]
For $A_2$ and $A_3$, we have
\begin{align*}
A_2+A_3=&\  \sum_{i=0}^{N-1}\Big(\left(F(X_{t_{i+1}})-F(X_{t_{i+1-}}) \right)+\left(F(X_{t_{i+}})-F(X_{t_{i}}) \right)\Big)\\
&+\sum_{i=0}^{N-1}\sum_{k=1}^{\left \lfloor p \right \rfloor }-\frac{1}{k!}D^kF(X_{t_{i}})(\Delta_{t_i}^+X)^{\otimes k}+\sum_{i=0}^{N-1}\sum_{k=1}^{\left \lfloor p \right \rfloor }-\frac{1}{k!}D^kF(X_{t_{i+1}-})(\Delta_{t_{i+1}}^-X)^{\otimes k}\\
=&\ \sum_{i=0}^{N-1}\Big(F(X_{t_{i+1}})-F(X_{t_{i+1-}})-\sum_{k=1}^{\left \lfloor p \right \rfloor }\frac{1}{k!}D^kF(X_{t_{i+1}-})(\Delta_{t_{i+1}}^-X)^{\otimes k} \Big) \\
&+\sum_{i=0}^{N-1}\Big(F(X_{t_{i+}})-F(X_{t_{i}})-\sum_{k=1}^{\left \lfloor p \right \rfloor }\frac{1}{k!}D^kF(X_{t_{i}})(\Delta_{t_i}^+X)^{\otimes k} \Big)\\
=&\ \sum_{0<t\le T}\Big( F(X_t) - F(X_{t-}) - \sum_{k=1}^{\lfloor p \rfloor} \frac{1}{k!} D^k F(X_{t-})(\Delta_t^- X)^{\otimes k}\Big) \\
&+\sum_{0\le t< T}\Big(F(X_{t+}) - F(X_t) - \sum_{k=1}^{\lfloor p \rfloor} \frac{1}{k!} D^k F(X_{t})(\Delta_t^+ X)^{\otimes k}\Big).
\end{align*}

It remains to justify the passage from finite sums to the total jump corrections. First, we establish the absolute convergence of the jump correction series appearing in the main theorem. For any jump point $t \in (0, T]$, let $\mathcal{J}_t^-$ denote the left-jump correction term:
\[ \mathcal{J}_t^- := F(X_t) - F(X_{t-}) - \sum_{k=1}^{\lfloor p \rfloor} \frac{1}{k!} D^k F(X_{t-})(\Delta_t^- X)^{\otimes k}. \]
Since $F \in \mathcal{C}_b^{\lfloor p \rfloor+1}$, Taylor's formula implies $\|\mathcal{J}_t^-\| \le C \|\Delta_t^- X\|^{\lfloor p \rfloor + 1}$. Because $X$ has finite $p$-variation, the sum $\sum_{0 < t \le T} \|\Delta_t^- X\|^p$ is finite, which implies that $\sum_{0 < t \le T} \|\Delta_t^- X\|^{\lfloor p \rfloor + 1}$ converges absolutely since $\lfloor p \rfloor + 1 > p$. A similar argument holds for the right-jump terms $\mathcal{J}_t^+$, where $\mathcal{J}_t^+$ is defined by
\[ \mathcal{J}_t^+ := F(X_{t+}) - F(X_{t}) - \sum_{k=1}^{\lfloor p \rfloor} \frac{1}{k!} D^k F(X_{t})(\Delta_t^+ X)^{\otimes k}. \]

For a given partition $\Pi = \{0=t_0 < t_1 < \dots < t_N=T\}$, the terms $A_2$ and $A_3$ can be rearranged as follows:
\begin{align*}
A_2+A_3=&\  \sum_{i=0}^{N-1}\Big(\left(F(X_{t_{i+1}})-F(X_{t_{i+1-}}) \right)+\left(F(X_{t_{i+}})-F(X_{t_{i}}) \right)\Big)\\
&+\sum_{i=0}^{N-1}\sum_{k=1}^{\left \lfloor p \right \rfloor }-\frac{1}{k!}D^kF(X_{t_{i}})(\Delta_{t_i}^+X)^{\otimes k}+\sum_{i=0}^{N-1}\sum_{k=1}^{\left \lfloor p \right \rfloor }-\frac{1}{k!}D^kF(X_{t_{i+1}-})(\Delta_{t_{i+1}}^-X)^{\otimes k}\\
=&\ \sum_{i=0}^{N-1}\Big(F(X_{t_{i+1}})-F(X_{t_{i+1-}})-\sum_{k=1}^{\left \lfloor p \right \rfloor }\frac{1}{k!}D^kF(X_{t_{i+1}-})(\Delta_{t_{i+1}}^-X)^{\otimes k} \Big) \\
&+\sum_{i=0}^{N-1}\Big(F(X_{t_{i+}})-F(X_{t_{i}})-\sum_{k=1}^{\left \lfloor p \right \rfloor }\frac{1}{k!}D^kF(X_{t_{i}})(\Delta_{t_i}^+X)^{\otimes k} \Big).
\end{align*}
As the partition $\Pi$ is refined in the RRS sense, we consider the limit over refinements. 
For any $\varepsilon> 0$, there are only finitely many jump points such that $\|\Delta_t^{-} X\|, \|\Delta_t^{+} X\| > \varepsilon$.
 By selecting a refinement $\Pi_\varepsilon$ that includes these ``large jumps" as endpoints, the contribution from ``small jumps" inside the partition intervals is controlled by the tail of the absolutely convergent series $\sum \|\Delta_t^{-} X\|^{\lfloor p \rfloor + 1}$, $\sum \|\Delta_t^{+} X\|^{\lfloor p \rfloor + 1}$. Therefore, the RRS limit of the finite sums yields the total jump corrections:
\begin{align*}
\text{RRS}-\lim_{{\Pi }} (A_2+A_3) =&\ \sum_{0<t\le T}\Big( F(X_t) - F(X_{t-}) - \sum_{k=1}^{\lfloor p \rfloor} \frac{1}{k!} D^k F(X_{t-})(\Delta_t^- X)^{\otimes k}\Big) \\
&+\sum_{0\le t< T}\Big(F(X_{t+}) - F(X_t) - \sum_{k=1}^{\lfloor p \rfloor} \frac{1}{k!} D^k F(X_{t})(\Delta_t^+ X)^{\otimes k}\Big).
\end{align*}

In conclustion, combining the limits of $A_1$, $A_2$, $A_3$, $B_1$, $B_2$, and $B_3$ yields the required~\eqref{it:bitem1}.
\end{proof}

\section{Stochastic and financial applications}\label{sec:4}

Theorem~\ref{thm:main} is entirely pathwise. Consequently, once a random signal has sample paths of finite $p$-variation, the pathwise It\^o-type formula can be applied path by path. This simple observation makes the theorem particularly well suited to stochastic models with jumps: the probabilistic input is carried by the sample-path regularity, while the analytic structure is still governed by the reduced rough integral from Proposition~\ref{prop:2}.

In this section we illustrate this point in several directions. We first state an almost sure version of Theorem~\ref{thm:main}. We then discuss two representative classes of stochastic drivers, namely pure-jump L\'evy-type models and mixed rough Gaussian-jump signals. After that, we explain how the formula propagates through random differential equations. We conclude with a financial application that gives a pathwise decomposition of cumulative log-wealth.

\subsection{An almost sure version of the main theorem}\label{subsec:stochastic_version}

We begin with the most direct stochastic consequence of our result. Since Theorem~\ref{thm:main} is deterministic, no new proof is needed once the finite $p$-variation property is known almost surely.

\begin{proposition}\label{prop:as_stochastic_formula}
Let $(\Omega,\mathcal{F},\mathbb{P})$ be a probability space, and let $X:[0,T]\times\Omega\to V$ be a c\`adl\`ag stochastic process. Assume that, for some $p\geq 1$,
\[
\|X(\omega)\|_{p\text{-var}; [0, T]}<\infty
\]
for almost every $\omega\in\Omega$. Let $F$ satisfy the regularity hypothesis in Theorem~\ref{thm:main}. Then, for almost every $\omega\in\Omega$,
\[
F(X_T)-F(X_0)
=
\int_0^T DF(X_t)\,d\X_t
+\sum_{0<t\leq T}\left(
F(X_t)-F(X_{t-})-\sum_{k=1}^{\lfloor p\rfloor}\frac{1}{k!}D^kF(X_{t-})(\Delta_t^-X)^{\otimes k}
\right),
\]
where $\Delta_t^-X=X_t-X_{t-}$, and the integral is understood in the RRS sense of Proposition~\ref{prop:2}.
\end{proposition}

\begin{proof}
Apply the c\`adl\`ag case of Theorem~\ref{thm:main} to each trajectory on the full-measure set where $X$ has finite $p$-variation.
\end{proof}

Proposition~\ref{prop:as_stochastic_formula} shows that the passage from deterministic paths to stochastic processes is immediate once one works path by path. In particular, when $X$ has continuous sample paths, the jump sum disappears and we recover the continuous It\^o-type formula. Moreover, whenever all terms are integrable, one may take expectations in Proposition~\ref{prop:as_stochastic_formula} to obtain a weak form of the identity.

\subsection{Representative stochastic drivers}\label{subsec:stochastic_drivers}

We next record two basic classes of stochastic processes to which Proposition~\ref{prop:as_stochastic_formula} applies. The first class concerns pure-jump finite-variation models, while the second highlights the role of arbitrary regularity in mixed rough systems.

\begin{proposition}\label{prop:levy_driver_application}
Let $X$ be a c\`adl\`ag L\'evy process with no Gaussian part and with L\'evy measure $\nu$ satisfying
\[
\int_{|z|\leq 1}|z|\,\nu(dz)<\infty.
\]
Then almost every sample path of $X$ has finite $1$-variation on $[0,T]$. Consequently, for every $F$ satisfying the regularity hypothesis in Theorem~\ref{thm:main},
\[
F(X_T)-F(X_0)
=
\int_0^T DF(X_t)\,d\X_t
+\sum_{0<t\leq T}\Big(
F(X_t)-F(X_{t-})-DF(X_{t-})\Delta_t^-X
\Big)
\]
almost surely.
\end{proposition}

\begin{proof}
Under the stated assumption, $X$ is a finite-variation L\'evy process (see, e.g.,~\cite{Sa99}), hence its sample paths have finite $1$-variation almost surely on compact intervals. The conclusion follows by applying the c\`adl\`ag case of Theorem~\ref{thm:main} (as stated in Proposition~\ref{prop:as_stochastic_formula}) with $p=1$.
\end{proof}

Proposition~\ref{prop:levy_driver_application} already covers many pure-jump models used in finance, insurance, and queueing theory. In particular, for a compound Poisson process the jump correction is simply a finite random sum over jump times.

The next example is more revealing from the point of view of roughness. It shows that the arbitrary-regularity extension proved in this paper is genuinely useful in stochastic settings where the continuous part of the signal is rough.

\begin{proposition}\label{prop:mixed_driver_application}
Let $B^H$ be a fractional Brownian motion with Hurst parameter $H\in(0,1)$, and let $J$ be a c\`adl\`ag stochastic process such that
\[
\|J\|_{q\text{-var};[0,T]}<\infty
\]
almost surely for some $q\geq 1$. Define
\[
X_t:=X_0+B_t^H+J_t,\quad \forall t\in [0, T].
\]
Then, for every $p>\max\{1/H,q\}$, almost every sample path of $X$ has finite $p$-variation on $[0,T]$. Hence, for every such $p$ and every $F$ satisfying the regularity hypothesis in Theorem~\ref{thm:main},
\[
F(X_T)-F(X_0)
=
\int_0^T DF(X_t)\,d\X_t
+\sum_{0<t\leq T}\left(
F(X_t)-F(X_{t-})-\sum_{k=1}^{\lfloor p\rfloor}\frac{1}{k!}D^kF(X_{t-})(\Delta_t^-X)^{\otimes k}
\right)
\]
almost surely.
\end{proposition}

\begin{proof}
Almost every trajectory of $B^H$ has finite $p$-variation for every $p>1/H$ (see, e.g.,~\cite[Chapter 10]{FH20} or~\cite{Ly98}). Since $J$ has finite $q$-variation almost surely, it also has finite $p$-variation for every $p>q$. Therefore $X=B^H+J$ has finite $p$-variation almost surely for every $p>\max\{1/H,q\}$. The conclusion then follows from Proposition~\ref{prop:as_stochastic_formula}.
\end{proof}

When $H=\frac12$, Proposition~\ref{prop:mixed_driver_application} includes Brownian motion with jumps for any $p>2$. More importantly, when $H\leq \frac13$, every admissible choice of $p$ satisfies $p>3$, so this example lies beyond the regime $2\leq p<3$ considered in \cite{FZ18}. This illustrates why the arbitrary-regularity framework of Theorem~\ref{thm:main} is not merely a technical extension.

\subsection{Random differential equations with jumps}\label{subsec:random_equations}

The previous subsection dealt with the driving processes themselves. We now turn to nonlinear dynamics driven by such processes.
This is the setting in which a pathwise It\^o-type formula is especially useful, since one can pass from the driver to observables of the solution without leaving the same deterministic framework.

Consider the random differential equation
\begin{equation}
dY_t=V(Y_t)\,dX_t,\qquad Y_0=y_0,
\label{eq:rde}
\end{equation}
where $X$ is a c\`adl\`ag stochastic process whose sample paths have finite $p$-variation almost surely. Assume that, for almost every trajectory, the equation admits a c\`adl\`ag solution $Y$ with finite $p$-variation on $[0,T]$.

\begin{proposition}\label{prop:random_equation_application}
Let $Y$ be a c\`adl\`ag solution of the random differential equation~(\ref{eq:rde}) such that $\|Y(\omega)\|_{p\text{-var};[0,T]} < \infty$ for almost every $\omega \in \Omega$. Then, for every observable $F$ satisfying the regularity hypothesis in Theorem~\ref{thm:main}, the following identity holds almost surely:
\[
F(Y_T)-F(Y_0)
=
\int_0^T DF(Y_t)\,d\Y_t
+\sum_{0<t\leq T}\left(
F(Y_t)-F(Y_{t-})-\sum_{k=1}^{\lfloor p\rfloor}\frac{1}{k!}D^kF(Y_{t-})(\Delta_t^-Y)^{\otimes k}
\right).
\]
Furthermore, whenever the substitution rule is available for the chosen solution concept, one has
\[
\int_0^T DF(Y_t)\,d\Y_t
=
\int_0^T DF(Y_t)V(Y_t)\,d\X_t.
\]
\end{proposition}

\begin{proof}
Since $Y$ is a c\`adl\`ag process with finite $p$-variation almost surely, the conclusion follows immediately from Proposition~\ref{prop:as_stochastic_formula}.
\end{proof}

\begin{remark}
It should be emphasized that the validity of the substitution rule
$$\int_{0}^{T} DF(Y_t)\,d\Y_t = \int_{0}^{T} DF(Y_t) V(Y_t)\,d\X_t$$
in Proposition~\ref{prop:random_equation_application} is contingent upon the chosen solution theory for RDEs with jumps. In the rough path setting, different solution concepts, such as Marcus-type geometric solutions and controlled rough path solutions, together with their corresponding rough integral constructions, may lead to variations in the form of the substitution formula. While Theorem~\ref{thm:main} provides a general pathwise algebraic decomposition, its application to specific dynamical systems requires ensuring that the regularity of the vector field $V$ and the lifting of the rough path are consistent with the RRS integration framework established in Proposition~\ref{prop:2}. For a thorough discussion on the consistency between substitution rules and solution concepts in the jump setting, we refer the reader to Friz and Zhang~\cite{FZ18} and the references therein.
\end{remark}

Proposition~\ref{prop:random_equation_application} provides a chain rule for nonlinear observables of random systems with jumps. Combined with Subsection~\ref{subsec:stochastic_drivers}, it yields a unified pathwise treatment of L\'evy-driven equations, mixed fractional-jump systems, and other non-semimartingale models.

\subsection{A financial application: pathwise log-wealth decomposition}\label{subsec:finance_application}

We finally discuss a simple application to mathematical finance. The key point is that Theorem~\ref{thm:main} converts multiplicative price or wealth dynamics into an additive decomposition for logarithmic returns, while keeping the jump correction explicit.

Let $W$ be a positive c\`adl\`ag wealth process with finite $p$-variation almost surely. Assume that, almost surely, the range of $W$ on $[0,T]$ is contained in a compact interval $[m,M]\subseteq (0,\infty)$. Choose a function $\varphi$ satisfying the regularity hypothesis in Theorem~\ref{thm:main} such that $\varphi(x)=\log x$ on $[m,M]$. Applying Theorem~\ref{thm:main} to $\varphi(W)$ yields the following identity.

\begin{proposition}\label{prop:finance_application}
Under the above assumptions,
\[
\log W_T-\log W_0
=
\int_0^T \frac{1}{W_t}\,d{\bf W}_t
+\sum_{0<t\leq T}\left(
\log\Bigl(1+\frac{\Delta_t^-W}{W_{t-}}\Bigr)-\sum_{k=1}^{\lfloor p\rfloor}\frac{(-1)^{k-1}}{k}\Bigl(\frac{\Delta_t^-W}{W_{t-}}\Bigr)^k
\right)
\]
almost surely, where the integral is the reduced rough integral defined in Proposition~\ref{prop:2}.
\end{proposition}

\begin{proof}
Apply Theorem~\ref{thm:main} to $\varphi$. Since $\varphi=\log$ on the range of $W$, we may replace $\varphi$ and its derivatives by $\log$ and its derivatives along the path. The identity then follows from
\[
\frac{d^k}{dx^k}\log x = \frac{(-1)^{k-1}(k-1)!}{x^k},\qquad k\geq 1. \qedhere
\]
\end{proof}

A particularly transparent case is obtained when the wealth process is driven by a return signal $R$ through the multiplicative equation
\[
dW_t=W_{t-}\pi_{t-}\,dR_t,\qquad W_0>0,
\]
where $R$ is a c\`adl\`ag process of finite $p$-variation and $\pi$ is the trading strategy. If the reduced rough integral $\int_0^T \pi_{t-}\,dR_t$ is well defined, the substitution rule holds, and $1+\pi_{t-}\Delta_t^-R>0$ at every jump time, then Proposition~\ref{prop:finance_application} becomes
\[
\log W_T-\log W_0
=
\int_0^T \pi_{t-}\,d{\bf R_t}
+\sum_{0<t\leq T}\left(\log(1+\pi_{t-}\Delta_t^-R)-\sum_{k=1}^{\lfloor p\rfloor}\frac{(-1)^{k-1}}{k}(\pi_{t-}\Delta_t^-R)^k\right).
\]
This formula may be read as a pathwise decomposition of cumulative log-return into a reduced rough integral term and an explicit jump correction. In the finite-variation case $p=1$, the correction reduces to the familiar first-order compensation $\log(1+\pi_{t-}\Delta_t^-R)-\pi_{t-}\Delta_t^-R$. For rougher return signals, the same interpretation remains valid after adding the higher-order terms prescribed by Theorem~\ref{thm:main}.

To summarize, the results of Subsections~\ref{subsec:stochastic_version}--\ref{subsec:finance_application} show that our main theorem is not only a pathwise extension of earlier It\^o-type formulas, but also a practical tool for stochastic analysis. It applies directly to random drivers with jumps, propagates through random differential equations, and yields explicit identities for quantities of interest in applications such as cumulative returns and log-wealth.

\vskip 0.2in

\noindent
{\bf Acknowledgments.} This work is supported by the National Natural Science Foundation of China (12571019), the Natural Science Foundation of Gansu Province (25JRRA644) and Innovative Fundamental Research Group Project of Gansu Province (23JRRA684).

\noindent
{\bf Declaration of interests. } The authors have no conflicts of interest to disclose.

\noindent
{\bf Data availability. } Data sharing is not applicable as no new data were created or analyzed.

\end{document}